\documentclass{article}
\usepackage{amsfonts,amsmath,amsthm,amssymb}
\usepackage{graphics, epsfig}
\usepackage{color}
\usepackage{appendix}
\usepackage{ulem}
\usepackage[makeroom]{cancel}
\newenvironment{ack}{\medskip{\it Acknowledgement.}}{}

 \usepackage[usenames,dvipsnames]{pstricks}
 \usepackage{pst-grad} 
 \usepackage{pst-plot} 
\allowdisplaybreaks

\let\TeXchi\chi
\newbox\chibox
\setbox0 \hbox{\mathsurround0pt $\TeXchi$}
\setbox\chibox \hbox{\raise\dp0 \box 0 }
\def\chi{\copy\chibox}

\newtheorem{proposition}{Proposition}[section]
\newtheorem{theorem}{Theorem}[section]
\newtheorem{lemma}{Lemma}[section]
\newtheorem{corollary}{Corollary}[section]
\newtheorem{remark}{Remark}[section]

\numberwithin{equation}{section}
\numberwithin{theorem}{section}
\numberwithin{proposition}{section}
\numberwithin{lemma}{section}
\numberwithin{remark}{section}
\newcommand{\noi}{\noindent}

\newcommand{\dsty}{\displaystyle}
\newcommand{\txty}{\textstyle}

\newcommand{\al}{\alpha}

\newcommand{\gm}{\gamma}
\newcommand{\dl}{\delta}

\newcommand{\lm}{\lambda}
\newcommand{\Lm}{\Lambda}

\newcommand{\vp}{\varphi}
\newcommand{\sig}{\sigma}

\newcommand{\om}{\omega}
\newcommand{\Om}{\Omega}

\newcommand{\z}{\zeta}

\newcommand{\df}[1]{\buildrel\mbox{\small def}\over{#1}}

\newcommand{\rr}{\mathbb{R}}
\newcommand{\rn}{\rr^N}

\newcommand{\bl}[1]{\mathbf{#1}}

\newcommand{\dvg}{\operatorname{div}}

\newcommand{\osc}{\operatornamewithlimits{osc}}

\newcommand{\loc}{\operatorname{loc}}

\newcommand{\pl}{\partial}

\newcommand{\intl}{\int\limits}

\def\Xint#1{\mathchoice
    {\XXint\displaystyle\textstyle{#1}}%
    {\XXint\textstyle\scriptstyle{#1}}%
    {\XXint\scriptstyle\scriptscriptstyle{#1}}%
    {\XXint\scriptscriptstyle\scriptscriptstyle{#1}}%
    \!\int}
\def\XXint#1#2#3{\setbox0=\hbox{$#1{#2#3}{\int}$}
    \vcenter{\hbox{$#2#3$}}\kern-0.5\wd0}
\def\bint{\Xint-}
\def\dashint{\Xint{\raise4pt\hbox to7pt{\hrulefill}}}
\def\dashiint{\bint\kern-0.15cm\bint}

\newcommand{\ovl}[3]{\int_{#1}^{#2}\kern-#3pt\raise4pt\hbox to7pt{\hrulefill}\ }

\newcommand{\ovll}[3]{\intl_{#1}^{#2}\kern-#3pt\raise4pt\hbox to7pt{\hrulefill}\ }

\newcommand{\tvl}[2]{\iint_{#1}\kern-#2pt\raise4pt\hbox to7pt{\hrulefill}\ }

\newcommand{\bye}{\end{document}}

\input harnack_mono.mac
\begin{document}
\title{A Boundary Estimate for Singular Parabolic Diffusion Equations}
\author
{Ugo Gianazza\\
Dipartimento di Matematica ``F. Casorati", 
Universit\`a di Pavia\\ 
via Ferrata 1, 27100 Pavia, Italy\\
email: {\tt gianazza@imati.cnr.it}
\and
Naian Liao\thanks{Corresponding author}\\
College of Mathematics and Statistics\\
Chongqing University\\
Chongqing, China, 401331\\
email: {\tt liaon@cqu.edu.cn}
\and
Teemu Lukkari\\
Department of Mathematics\\
P.O. Box 11100, 00076 Aalto University\\
Espoo, Finland\\
email: {\tt teemu.lukkari@aalto.fi}}
\date{}
\maketitle
\vskip.4truecm
\centerline{\it Dedicated to Emmanuele DiBenedetto
for his 70$^{th}$ birthday}
\vskip.4truecm
\begin{abstract}
We prove an estimate on the modulus of continuity at a boundary point of a cylindrical domain for local weak solutions to singular parabolic equations of $p$-laplacian type. The estimate is given in terms of a Wiener-type integral, defined by a proper elliptic $p$-capacity. 
\vskip.2truecm
\noindent{\bf AMS Subject Classification (2010):} 
Primary 35K67, 35B65; Secondary 35B45, 35K20
\vskip.2truecm
\noindent{\bf Key Words}: Parabolic $p$-laplacian, boundary estimates, continuity, elliptic $p$-capacity, Wiener-type integral.
\end{abstract}
\section{Introduction}\label{S:intro}
Let $E$ be an open set in $\rn$ and for $T>0$ let $E_T$ denote the
cylindrical domain $E\times(0,T]$. Moreover let
\begin{equation*}
S_T=\partial E\times[0,T],\qquad \partial_p E_T=S_T\cup(E\times\{0\})
\end{equation*}
denote the lateral, and the parabolic boundary respectively.

We shall consider quasi-linear, parabolic partial differential equations of the form
\begin{equation}  \label{Eq:1:1}
u_t-\dvg\bl{A}(x,t,u, Du) = 0\quad \text{ weakly in }\> E_T
\end{equation}
where the function $\bl{A}:E_T\times\rr^{N+1}\to\rn$ is only assumed to be
measurable and subject to the structure conditions
\begin{equation}  \label{Eq:1:2}
\left\{
\begin{array}{l}
\bl{A}(x,t,u,\xi)\cdot \xi\ge C_o|\xi|^p \\
|\bl{A}(x,t,u,\xi)|\le C_1|\xi|^{p-1}%
\end{array}%
\right .\quad \text{ a.e.}\> (x,t)\in E_T,\, \forall\,u\in\rr,\,\forall\xi\in\rn
\end{equation}
where $C_o$ and $C_1$ are given positive constants, and $1<p<2$.

The 
principal part $\bl{A}$ is required to be monotone in the variable $\xi$ 
in the sense
\begin{equation}\label{Eq:5:1:3}
(\bl{A}(x,t,u,\xi_1)-\bl{A}(x,t,u,\xi_2))
\cdot(\xi_1-\xi_2)\ge0
\end{equation}
for all variables in the indicated domains and Lipschitz continuous 
in the variable $u$, that is,
\begin{equation}\label{Eq:5:1:4}
|\bl{A}(x,t,u_1,\xi)-\bl{A}(x,t,u_2,\xi)|\le \Lm|u_1-u_2| (1+|\xi|^{p-1})
\end{equation}
for some given $\Lm>0$, and for the variables in the indicated domains.

We refer to the parameters $\datap$ as our structural data, and we write $\gm%
=\gm(p,N,C_o,C_1)$ if $\gm$ can be quantitatively
determined a priori only in terms of the above quantities.
A function
\begin{equation}  \label{Eq:1:4}
u\in C\big(0,T;L^2_{\loc}(E)\big)\cap L^p_{\loc}\big(0,T; W^{1,p}_{%
\loc}(E)\big)
\end{equation}
is a local, weak sub(super)-solution to \eqref{Eq:1:1}--\eqref{Eq:1:2} if
for every compact set $K\subset E$ and every sub-interval $[t_1,t_2]\subset
(0,T]$
\begin{equation}  \label{Eq:1:5}
\int_K u\vp dx\bigg|_{t_1}^{t_2}+\int_{t_1}^{t_2}\int_K \big[-u\vp_t+\bl{A}%
(x,t,u,Du)\cdot D\vp\big]dxdt\le(\ge)0
\end{equation}
for all non-negative test functions
\begin{equation*}
\vp\in W^{1,2}_{\loc}\big(0,T;L^2(K)\big)\cap L^p_{\loc}\big(0,T;W_o^{1,p}(K)%
\big).
\end{equation*}
This guarantees that all the integrals in \eqref{Eq:1:5} are convergent.

For any $k\in\rr$, let
\[
(v-k)_-=\max\{-(v-k),0\},\qquad(v-k)_+=\max\{v-k,0\}.
\]
We require \eqref{Eq:1:1}--\eqref{Eq:1:2} to be parabolic, namely that whenever $u$ is a weak solution, for all $k\in\rr$, the functions $(u-k)_\pm$ are weak sub-solutions, with $\bl{A}(x,t,u,Du)$ replaced by $\pm\bl{A}(x,t,k\pm(u-k)_\pm,\pm D(u-k)_\pm)$. As discussed in condition ({\bf A}$_6$) of \cite[Chapter II]{dibe-sv} or Lemma~1.1 of \cite[Chapter 3]{DBGV-mono}, such a condition is satisfied, if for all {$(x,t,u)\in E_T\times\rr$} we have
\[
\bl{A}(x,t,u,\eta)\cdot\eta\ge0\qquad\forall\,\eta\in\rn,
\]
which we assume from here on.


For $y\in \rn$ and $\rho > 0 $, $K_{\rho}(y)$ denotes the cube of edge $2\rho$, centered at $y$
with faces parallel to the coordinate planes. When $y$ is the
origin of $\rn$ we simply write $K_{\rho}$.

For $a,\,\theta_1,\,\theta_2 > 0$ and $(y,s)\in E_T$, we will consider 
\[
\text{ forward  cylinders: }\ \  K_{a\rho} (y) \times (s,s+\theta_2 \rho^p],
\]
\[
\text { backward cylinders: }\ \  K_{a\rho} (y) \times (s-\theta_1 \rho^p,s],
\]
\[
\text{ centered cylinders: }\ \  K_{a\rho} (y) \times (s-\theta_1 \rho^p,s+\theta_2\rho^p].
\]
We are interested in the boundary behaviour of solutions to the Cauchy-Dirichlet problem
\begin{equation}\label{Eq:1:6}
\left\{
\begin{aligned}
&u_t-\dvg\bl{A}(x,t,u, Du) = 0\quad \text{ weakly in }\> E_T\\
&u(\cdot,t)\Big|_{\partial E}=g(\cdot,t)\quad \text{ a.e. }\ t\in(0,T]\\
&u(\cdot,0)=g(x,0),
\end{aligned}
\right.
\end{equation}
where 
\begin{itemize}
\item {\bf (H1)}: $\bl{A}$ satisfies \eqref{Eq:1:1}--\eqref{Eq:5:1:4} for $\frac{2N}{N+1}<p<2$, 
\item {\bf (H2)}: $\dsty g\in L^p(0,T;W^{1,p}( E))$, and $g$ is
  continuous on $\overline{E}_T$ with modulus of continuity
  $\om_g(\cdot)$.
\end{itemize}
We do not impose any {\it a priori} requirements on the boundary of the domain
$E\subset\rn$.


A weak sub(super)-solution to
the Cauchy-Dirichlet problem \eqref{Eq:1:6} is a measurable function $u\in C\big(0,T;L^2(E)\big)\cap 
L^p\big(0,T; W^{1,p}(E)\big)$ satisfying
\begin{equation} \label{Eq:1:7}
\begin{aligned}
&\int_E u\vp(x,t) dx+\iint_{E_T} \big[-u\vp_t+\bl{A}%
(x,t,u,Du)\cdot D\vp\big]dxdt\\
&\le(\ge)\int_E g\vp(x,0) dx
\end{aligned}
\end{equation}
for all non-negative test functions
\begin{equation*}
\vp\in W^{1,2}\big(0,T;L^2(E)\big)\cap L^p\big(0,T;W_o^{1,p}(E)\big).
\end{equation*}
In addition, we take the boundary condition $u\le g$ ($u\ge g$) to
mean that $(u-g)_+(\cdot,t)\in W^{1,p}_o(E)$ ($(u-g)_-(\cdot,t)\in W^{1,p}_o(E)$) for
a.e. $t\in(0,T]$. A function $u$ which is both a weak sub-solution and
a weak super-solution is a solution. Notice that the range we are
assuming for $p$, and the continuity of $g$ on the closure of $E_T$
ensure that a weak solution $u$ to \eqref{Eq:1:6} is bounded.

In the sequel we will need {the following comparison principle for weak (sub/super)-solutions
(see \cite[Lemma~3.1]{KiLi96} and
  \cite[Lemma~3.5]{KoKuPa10})}.
The lower semicontinuity of weak super-solutions is discussed in \cite{kuusi2009}.
\begin{lemma}[Weak Comparison Principle]\label{lem-comp}
Suppose that $u$ is a weak super-solution and $v$ is 
a weak sub-solution to \eqref{Eq:1:6}
in $E_T$. If $u$ and $-v$ are lower semicontinuous on $\overline{E_T}$ and $v\le u$ on the parabolic boundary $\partial_p E_T$, then $v\le u$ a.e.\ in $E_T$.
\end{lemma}
Although the definition has been given in a global way, all the
following arguments and results will have a local thrust: indeed, what we
are interested in, is whether solutions $u$ to
\eqref{Eq:1:1}--\eqref{Eq:1:2} continuously assume the given boundary
data in a single point or some other distinguished part of the lateral
boundary $S_T$ of a cylinder, but not necessarily on the whole $S_T$.
In this context the initial datum will play no role.

Let $\pto\in S_T$, and for $R_o>0$ set
\[
Q_{R_o}=K_{R_o}(x_o)\times(t_o-2R_o^p,t_o+R_o^p],
\]
where $R_o$ is so small that $(t_o-2R_o^p,t_o+R_o^p]\subset(0,T]$.
Moreover, set
\[
\mu_o^+=\sup_{Q_{R_o}\cap E_T}u,\qquad\mu_o^-=\inf_{Q_{R_o}\cap E_T}u,\qquad\om_o=\mu_o^+-\mu_o^-=\osc_{Q_{R_o}\cap E_T}u. 
\]
Here and in the following $c\in(0,1)$ is the same constant as in \eqref{Eq:3:7:2} below. We can then state the main result of this paper.

\begin{theorem}\label{Thm:1:1}
Let $u$ be a weak solution to \eqref{Eq:1:6}, and assume that {\bf (H1)} and {\bf (H2)} hold true. Then there exist  two positive constants $\gm$ and $\al$, that depend only on the data $\datap$, such that
\begin{equation}\label{Eq:1:9}
\osc_{Q_{\rho}(\om_o)\cap E_T}u\le
\om_o\exp\left\{-{\gm}\int_{\rho^\al}^1 \left[\dl(s)\right]^{\frac1{p-1}}\frac{ds}{s}\right\}
+2\osc_{\tilde{Q}_o(\rho)\cap S_T}g,
\end{equation}
%
%
%
where $0<\rho<R_o$ and 
\[
Q_\rho(\om_o)=K_\rho(x_o)\times[t_o-\frac c4\om_o^{2-p}2\rho^p, t_o+\frac c4\om_o^{2-p}\rho^p],
\]
\[
\dl(\rho)=\frac{{\rm cap}_p(K_{\rho}(x_o)\backslash E,K_{\frac32\rho}(x_o))}{{\rm cap}_p(K_{\rho}(x_o),K_{\frac32\rho}(x_o))},
\]
$\tilde{Q}_o(\rho)$ is a proper reference cylinder (see \eqref{large:cyl} below), and ${\rm cap}_p(D,B)$ denotes the (elliptic) $p$-capacity of $D$ with respect to $B$.
\end{theorem}
A point $(x_o,t_o)\in S_T$ is called a {\it Wiener point} if
$\dsty\int_\tau^1\left[\dl(\rho)\right]^{\frac1{p-1}}\frac{d\rho}\rho\to\infty$
as $\tau\to0$. Therefore, {from Theorem~\ref{Thm:1:1} we can conclude
  the following corollary in a standard way.}
\begin{corollary}
Let $u$ be a weak solution to \eqref{Eq:1:6}, assume that {\bf (H1)} and {\bf (H2)} hold true, and that $(x_o,t_o)\in S_T$ is a Wiener point. Then 
\[
\lim_{\genfrac{}{}{0pt}{}{(x,t)\to(x_o,t_o)}{(x,t)\in E_T}}u(x,t)=g(x_o,t_o).
\]
\end{corollary}

Theorem \ref{Thm:1:1} also implies H\"older regularity up to the
boundary under a fairly weak assumption on the domain. More
spesifically, a set $A$ is uniformly $p$-fat, if for some
$\gm_o,\,\rho_o>0$ one has
\[
\frac{{\rm cap}_p(K_{\rho}(x_o)\cap A,K_{\frac32\rho}(x_o))}{{\rm
    cap}_p(K_{\rho}(x_o),K_{\frac32\rho}(x_o))}\geq \gm_o
\]
for all $0<\rho <\rho_o$ and all $x_o\in A$. See \cite{lewis1988} for
more on this notion. We have the following corollary. 

\begin{corollary}\label{cor:holder}
  Let $u$ be a weak solution to \eqref{Eq:1:6}, assume that {\bf (H1)}
  and {\bf (H2)} hold true, that the complement of the domain $E$ is
  uniformly $p$-fat, and let $g$ be H\"older continuous.  Then the
  solution $u$ is H\"older continuous up to the boundary.
\end{corollary}
\subsection{Novelty and Significance}
For solutions to {\it linear, elliptic} equations with bounded and measurable coefficients, vanishing on a part of the boundary $\partial E$ near the origin, it is well-known that a weak solution $u$ satisfies the following estimate 
\begin{equation}\label{Eq:1:8}
\sup_{E\cap B_r}|u(x)|\le C\sup_{E\cap B_{r_o}}|u(x)|\exp\left(-C'\int_r^{r_o}\frac{\hbox{cap}(E^c\cap B_\rho)}{\rho^{N-2}}\frac{d\rho}\rho\right).
\end{equation}
It was proved by Maz'ya for harmonic functions in \cite{mazya63} and for solutions to more general linear equations in \cite{mazya67}. Such a result implies the sufficiency part of the Wiener criterion. Similar estimates have then been obtained for solutions to equations of $p$-laplacian type in \cite{mazya76,Gar-Zie} (just to mention only few results; for a comprehensive survey of the elliptic theory, see \cite{mal-zie}).

In the parabolic setting, these estimates have been proved for the heat operator in \cite{biroli-mosco85}. Continuity at the boundary for quite general operators with the same growth has been considered in \cite{ziemer, ziemer1982}. {Parabolic quasi-minima are dealt with in \cite{marchi}: in such a case, no explicit estimate of the modulus of continuity is given, but the divergence of a proper Wiener integral yields the continuity of the quasi-minimum.} They have also been extended to the parabolic obstacle problem in \cite{biroli-mosco87}. 

To our knowledge, much less is known for nonlinear diffusion operators, such as the $p$-laplacian.
%
A result similar to ours, obtained with different techniques, is stated in \cite{skrypnik}. In such a case, the comparison principle plays a central role. Even though we also assume that the comparison principle is satisfied, nevertheless its role is limited to the proof of the weak Harnack inequality, whereas all the other estimates are totally independent of it. Therefore, if one could prove Proposition~\ref{Prop:3:1} for a general operator, then Theorem~\ref{Thm:1:1} and its corollaries would automatically hold too.

The fact that a Wiener point is a continuity point has already been
observed in \cite{BBGP} for the prototype parabolic $p$-laplacian, for
any $p>1$, hence both singular, i.e. with $1<p<2$, and degenerate,
i.e. with $p>2$ (see also \cite{KiLi96}). Under this point of view,
here the novelty is twofold: we deal with more general operators, and
we provide a quantitative modulus of continuity in terms of the
integral of the relative capacity $\dl(\rho)$. On the other hand, here
we focus just on the singular case, and, due to the use of the weak
Harnack inequality, we have to limit $p$ in the so-called singular
super-critical range $(\frac{2N}{N+1},2)$. The results in \cite{BBGP}
suggest that a statement like Theorem~\ref{Thm:1:1} should hold also
in the singular critical and sub-critical range
$1<p\le\frac{2N}{N+1}$; this is no surprise, since locally bounded
solutions are locally continuous {in} the interior for any $1<p<2$,
and there is no apparent reason, why things should be different, when
working at the boundary. This problem will be investigated in a future
paper, together with the degenerate case $p>2$.

Finally, Corollary~\ref{cor:holder} can be seen as an extension of Theorem~1.2 of \cite[Chapter IV]{dibe-sv}, where the H\"older continuity up to the boundary of weak solutions to the Cauchy-Dirichlet problem \eqref{Eq:1:6} with H\"older continuous boundary data is proved assuming that the domain $E$ satisfies a positive geometric density condition, namely that there exist $\al_*\in(0,1)$ and $\rho_o>0$ such that 
$\forall\,x_o\in\partial E$, for every ball $B_\rho(x_o)$ centered at $x_o$ and radius $\rho\le\rho_o$ 
\[
|E\cap B_\rho(x_o)|\le(1-\al_*)|B_\rho(x_o)|.
\]
It is not hard to see that if a domain $E$ has positive geometric
density, then the complement of $E$ is uniformly $p$-fat, but the
opposite implication obviously does not hold.


The proof of Theorem~\ref{Thm:1:1} is given Section~\ref{S:final}, whereas the previous sections are devoted to introductory material, namely a boundary $L^1$ Harnack inequality (Section~2), the weak Harnack inequality (Section~3), the definition of capacity (Section~4), and a final auxiliary lemma (Section~5).

\begin{ack} 
{\normalfont The authors thank Juha Kinnunen, who suggested this problem, during the program ``Evolutionary problems'' in the Fall 2013 at the Institut Mittag-Leffler, and are very grateful to Emmanuele DiBene\-det\-to, for discussions and comments, which greatly helped to improve the final version of this manuscript.}
\end{ack}
\section{A Boundary $L^1$ Harnack Inequality for Super-Solutions in the Whole Range 
$p\in(1,2)$}\label{S:A:1}

Fix $(x_o,t)\in S_T$, consider the cylinder 
\begin{equation}\label{Eq:cyl1}
Q=K_{16\rho}(x_o)\times[s,t],
\end{equation} 
where $s$, $t$ are such that $0<s<t<T$, and let $\Sigma\df=S_T\cap Q$. Our estimates are based on the
following simple lemma.

\begin{lemma}\label{Lem:2:0}
 Take any number $k$ such that
  $k\ge\sup_\Sigma g$. Let $u$ be a weak solution to the problem
  \eqref{Eq:1:6}, and define 
  \begin{displaymath}
    u_k=
    \begin{cases}
      (u-k)_+, &\text{in }Q\cap E_T, \\
      0, & \text{in }Q\setminus E_T.
    \end{cases}
  \end{displaymath}  
  Then $u_k$ is a (local) weak sub-solution in the cylinder $Q$. The
  same conclusion holds for the zero extension of $u_h=(h-u)_+$ for
  truncation levels $h\leq \inf_{\Sigma}g$.
\end{lemma}
\begin{proof}
  We first claim that $u_k(\cdot,\tau)\in W^{1,p}(K_{16\rho}(x_o))$
  for almost all $s<\tau<t$. We may assume that $k=\sup_\Sigma g$, as
  the general case follows by repeating a part of the argument using
  the fact that $(u-k)_+\leq (u-\sup_{\Sigma} g)_+$.

  Let $E_j$, $j=1,2,\ldots$ be open subsets exhausting $E$, i.e.
  $E_j\subset E_{j+1}\subset \ldots \subset E$ with compact inclusions,
  and $E=\cup_{j=1}^{\infty}E_j$.  We set
  \begin{displaymath}
    k_j=\sup_{Q\cap \overline{E}\times[s,t]\setminus E_j\times[s,t]}g,
  \end{displaymath}
  and note that $k_j\to k$ as $j\to \infty$ by the continuity of $g$,
  so it suffices to prove the claim for each $k_j$. 

  To proceed, pick $\varphi\in C_o^{\infty}(E)$ with $0\leq \varphi \leq 1$
  and $\varphi=1$ in $\overline{E}_j$, and define
  \begin{displaymath}
    w=(1-\varphi)(u-g)_++\varphi (u-k_j)_+.
  \end{displaymath}
  Then $(u-k_j)_+\leq w $ in $Q\cap E_T$, and $w(\cdot, \tau)\in
  W^{1,p}_o(E)$ for a.e. $\tau\in(s,t)$.  We may choose $\eta_l\in
  C_o^\infty(E)$ converging to $w(\cdot,\tau)$ in $W^{1,p}(E)$. By
  standard facts about first-order Sobolev spaces,
  $\min(u_{k_j},\eta_l)$ converges to $\min(u_{k_j},w)$ in $
  W^{1,p}(K_{16\rho}(x_o))$. Now, since $(u-k_j)_+\leq w$, the
  conclusion follows from the fact that $\min(u_{k_j},w)=u_{k_j}$.

  The fact that $u_k$ satisfies the integral inequality for
  sub-solutions in $Q$ follows by arguing as in pp. 18-19 of
  \cite{dibe-sv}.
\end{proof}

Let $k$ be any number such that $k\ge \sup_{\Sigma} g$, and for $u_k$ as in Lemma~\ref{Lem:2:0},
\begin{equation}\label{Eq:super}
\left\{
\begin{aligned}
&\mu=\sup_Q u_k,\\ 
&v:Q\rightarrow\rr_+,\quad v\df=\mu-u_k.
\end{aligned}
\right.
\end{equation}
It is not hard to verify that $v$ is a weak super-solution to
\eqref{Eq:1:6} in the whole $Q$. 

\begin{proposition}\label{Prop:A:1:1} 
Let $(x_o,t)\in S_T$, $s$ such that $0< s<t< T$,
and 
$Q$ as in \eqref{Eq:cyl1}.  There exists a positive constant 
$\gm$ depending only on the data $\datap$, such that 
\begin{equation}\label{Eq:A:1:2}
\sup_{s<\tau<t}\int_{K_\rho(x_o)} v(x,\tau)dx\le\gm
\int_{K_{2\rho}(x_o)}v(x,t)dx
+\gm\Big(\frac{t-s}{\rho^\lm}\Big)^{\frac1{2-p}}
\end{equation}
where 
\begin{equation*}
\lm=N(p-2)+p,
\end{equation*}
and $v$ has been defined in \eqref{Eq:super}. 
The constant $\gm=\gm(p)\to\infty$ either as $p\to2$ or as 
$p\to1$.
\end{proposition}
For $\lm>0$, the parameter $p$ is in the singular, 
supercritical range $(\frac{2N}{N+1},2)$, 
and if $\lm\le0$, $p$ is in the singular, critical 
and subcritical range $(1,\frac{2N}{N+1}]$.  
However, the Harnack-type estimate (\ref{Eq:A:1:2}), in the 
topology of $L^1_{\loc}$,  holds true for all $1<p<2$ 
and accordingly, $\lm$ could be of either sign. 
\begin{proof}
See Proposition~A.1.3 in \cite[Appendix~A]{DBGV-mono}. Here conditions~\eqref{Eq:5:1:3}--\eqref{Eq:5:1:4} are not needed in the proof.
\end{proof}
\noi We also need the following result, whose proof can be found as before in \cite[Appendix~A]{DBGV-mono}.
\begin{lemma}\label{Lm:A:1:2} 
Let $(x_o,t)\in S_T$, $s$ such that $0< s<t< T$,
and 
$v$ as defined in \eqref{Eq:super}. There exists a positive constant 
$\gm$ depending only on the data, such that for all $\sig\in(0,1)$, 
\begin{align*}
\frac1{\rho}\int_s^t\int_{K_{\sig\rho}(x_o)}
|D v|^{p-1}dx\,d\tau&\le\dl\sup_{s<\tau<t}\int_{K_\rho(x_o)} v(x,\tau)dx\\
&+\frac{\gm(p)}{[\dl^2(1-\sig)^{p}]^{\frac{p-1}{2-p}}}
\Big(\frac{t-s}{\rho^\lm}\Big)^{\frac1{2-p}}
\end{align*}
for all $\dl\in(0,1)$.
The constant $\gm(p)\to\infty$ as either $p\to 1,2$.\index{Stability of constants}
\end{lemma}
\begin{remark}\label{Rmk:2:1}
{\normalfont If we choose
\[
\bar s=t-\left[\int_{K_{2\rho}(x_o)}v(x,t)dx\right]^{2-p}\rho^\lm,
\]
and $\bar s>0$, then Proposition~\ref{Prop:A:1:1} yields
\[
\sup_{\bar s<\tau<t}\int_{K_\rho(x_o)}v(x,\tau)dx\le\gm\int_{K_{2\rho}(x_o)}v(x,t)dx.
\]}
\end{remark}
\begin{remark}\label{Rmk:2:2}
{\normalfont The choice of $k$ in the definition of $u_k$ is done in order to guarantee
that $u_k$ can be extended to zero in $Q\backslash E_T$: this yields a function which is defined on the whole $Q$, and is needed in Proposition~\ref{Prop:A:1:1}, and Lemmas~\ref{Lm:A:1:2} and \ref{Lm:5:1}. Therefore, any other choice of $k$ which ensures the same extension of $u$ to the whole $Q$ is allowed.}
\end{remark}
\noi Whenever we deal with a {\it solution}, and not just a super-solution, then 
it has been shown in Proposition~A.1.1 of \cite[Appendix~A]{DBGV-mono}
that the statement is much more general, as we have the following result.
\begin{proposition}\label{Prop:A:1:1bis} 
Let $u$ be a nonnegative, 
local, weak solution to the singular equations 
(\ref{Eq:1:1})--(\ref{Eq:1:2}), 
for $1<p<2$, in $E_T$.  There exists a positive constant 
$\gm$ depending only on the data $\datap$, such that 
for all cylinders $K_{2\rho}(y)\times[s,t]\subset E_T$, 
\begin{equation}\label{Eq:A:1:2bis}
\sup_{s<\tau<t}\int_{K_\rho(y)} u(x,\tau)dx\le\gm
\inf_{s<\tau<t}\int_{K_{2\rho}(y)}u(x,\tau)dx
+\gm\Big(\frac{t-s}{\rho^\lm}\Big)^{\frac1{2-p}}.
\end{equation}
The constant $\gm=\gm(p)\to\infty$ either as $p\to2$ or as 
$p\to1$.
\end{proposition}
\section{A Weak Harnack Inequality}
The Weak Harnack Inequality for non-negative super-solutions to equations
~\eqref{Eq:1:1}--\eqref{Eq:1:2} and $p>2$ has been proved in \cite{kuusi2008}. In the singular range $\frac{2N}{N+1}<p<2$, as we are considering here, to our knowledge it has been proved in \cite{BIV2010}, but only for the prototype parabolic $p$-laplacian.

Here we prove it under the more general conditions \eqref{Eq:1:2}, but we rely on the Comparison Principle, and therefore we also need \eqref{Eq:5:1:3}--\eqref{Eq:5:1:4}. It remains an open problem to prove the Weak Harnack Inequality in the range $\frac{2N}{N+1}<p<2$ for the general case, without relying on any kind of comparison. 

It is important to mention that the approach we follow here is closely modelled on the arguments of \cite{chen-dibe,fornaro-vespri}.
\begin{proposition}\label{Prop:3:1}
Let $u$ be a nonnegative, local, weak super-solution 
to equations \eqref{Eq:1:1}--\eqref{Eq:5:1:4}.  
There exist 
constants $c\in(0,1)$ and $\eta\in(0,1)$, depending only on the data 
$\datap$, such that for a.e. $s\in(0,T)$
\begin{equation}\label{Eq:3:7:1}
\inf_{K_{2\rho}(y)}u(\cdot,t)\ge \eta\bint_{K_{2\rho}(y)}u(x,s)dx
\end{equation}
for all times $\dsty s+{\txty\frac34}\theta\rho^p\le t\le s+\theta\rho^p$, where
\begin{equation}\label{Eq:3:7:2}
\theta\df=c\bigg[\dashint_{K_{2\rho}(y)}u(x,s)\,dx\bigg]^{2-p},
\end{equation}
provided that $\dsty K_{16\rho}(y)\times[s,s+\theta\rho^p]\subset E_T$.
\end{proposition}
%
\begin{proof}
Without loss of generality, we can assume $(y,s)=(0,0)$. 
Set $Q=K_{16\rho}\times(0,\infty)$ and let $v$ be the unique solution to the following problem.
\begin{equation*}
\left\{
\begin{aligned}
&v_t-\dvg \bl{A}(x,t,v,Dv)=0\quad\text{ in } Q;\\
&v(x,t)=0\quad\text{ on }\pl K_{16\rho}\times(0,\infty);\\
&v(x,0)=u(x,0)\chi_{K_{2\rho}}.
\end{aligned}
\right.
\end{equation*}
Then, Lemma~\ref{lem-comp} gives us
\begin{equation}\label{Eq:3:2}
u\ge v\qquad\text{ a.e. in }\ \ Q.
\end{equation}
Set
\begin{align*}
&\theta=\bigg[\dashint_{K_{8\rho}}v(x,0)\,dx\bigg]^{2-p}
=4^{-N(2-p)}\bigg[\dashint_{K_{2\rho}}u(x,0)\,dx\bigg]^{2-p},\\
&\dl=\frac1{(2\gm)^{2-p}}\theta\rho^p,
\end{align*}
where $\gm$ is the constant from Proposition~\ref{Prop:A:1:1bis}.
By Theorem 2.1 of \cite[Chapter~6]{DBGV-mono}, 
\[
\sup_{K_{4\rho}\times(\frac{\dl}2,\dl)}v\le\gm_1\dashint_{K_{2\rho}}u(\cdot,0)\,dx.
\]
On the other hand, Proposition~\ref{Prop:A:1:1bis} yields
\begin{align*}
\sup_{0<\tau<\dl}\int_{K_{2\rho}}v(x,\tau)\,dx
&\le\gm\inf_{0<\tau<\dl}\int_{K_{4\rho}}v(x,\tau)\,dx+\gm\bigg(\frac{\dl}{\rho^\lm}\bigg)^{\frac1{2-p}}\\
&\le\gm\inf_{0<\tau<\dl}\int_{K_{4\rho}}v(x,\tau)\,dx+\frac12\int_{K_{2\rho}}v(x,0)\,dx.
\end{align*}
This gives
\[
\inf_{0<\tau<\dl}\int_{K_{4\rho}}v(x,\tau)\,dx\ge\frac1{2\gm}\int_{K_{2\rho}}v(x,0)\,dx.
\]
Let $c_o>0$ to be determined. From the previous estimates, for any $\frac\dl2<\tau<\dl$, 
we have  
\begin{align*}
&\frac1{2^{N+1}\gm}\dashint_{K_{2\rho}}v(x,0)\,dx\le\dashint_{K_{4\rho}}v(x,\tau)\,dx\\
=&\frac1{|K_{4\rho}|}\int_{[v(x,\tau)>c_o]\cap K_{4\rho}}v(x,\tau)\,dx+\frac1{|K_{4\rho}|}\int_{[v(x,\tau)\le c_o]\cap K_{4\rho}}v(x,\tau)\,dx\\
\le&\frac{|[v(x,\tau)>c_o]\cap K_{4\rho}|}{|K_{4\rho}|}\gm_1\dashint_{K_{2\rho}}u(x,0)\,dx+c_o.
\end{align*}
Hence, if we take
\[c_o=\frac1{2^{N+2}\gm}\,\dashint_{K_{2\rho}}u(x,0)\,dx,\]
then
\[
|[v(x,\tau)>c_o]\cap K_{4\rho}|\ge\frac1{2^{N+2}\gm\gm_1}|K_{4\rho}|\quad\text{for any }\frac\dl2<\tau<\dl.
\]
By Proposition 5.1 on p.72 of \cite{DBGV-mono}, this information gives 
\[
\inf_{K_{8\rho}\times(\frac34\dl,\dl)}v\ge \eta\,\,\dashint_{K_{2\rho}}u(x,0)\,dx
\]
where the constant $\eta$ depends only on $\{p,N,C_o,C_1\}$.
By \eqref{Eq:3:2},
\[
\inf_{K_{8\rho}\times(\frac34\dl,\dl)}u\ge \eta\,\,\dashint_{K_{2\rho}}u(x,0)\,dx,
\]
and since
\[
\inf_{K_{2\rho}\times(\frac34\dl,\dl)}u\ge\inf_{K_{8\rho}\times(\frac34\dl,\dl)}u
\]
we conclude.
\end{proof}
\section{A Notion of Capacity}
Let $\Om\subset\rn$ be an open set, and $Q\df=\Om\times(t_1,t_2)$: $Q$ is an open cylinder in $\rr^{N+1}$. In the following we will refer to such sets as {\it open parabolic cylinders}. For any $K\subset Q$ compact, we define the {\it parabolic capacity of $K$ with respect to $Q$} as 
\begin{equation}\label{Eq:4:1}
\begin{aligned}
\gm_p(K,Q)=&\inf\left\{\iint_Q|D\vp|^p\,dxdt:\right.\\ 
&\vp\in C^\infty_o(Q),\ \ \vp\ge1\ \ \text{ on a neighborhood of }\ \ K\Big\}.
\end{aligned}
\end{equation}
For $p=2$, this notion of parabolic capacity has been introduced in \cite{biroli-mosco87} in order to study the decay of solutions to parabolic obstacle problems relative to second order, linear operators, and then applied to parabolic quasiminima in \cite{marchi}. 

It is important to remark that different kinds of nonlinear parabolic capacity have been recently introduced in \cite{KKKP} and in \cite{AKP}. We will not go into details here, and we will prove our estimates in terms of $\gm_p$ as defined in \eqref{Eq:4:1}. 

In the following, we will state the main properties of $\gm_p$. {For} the proofs, we refer to \cite[Appendix]{biroli-mosco87}, where detailed calculations are given for $p=2$, and they can be easily extended to the general context we are considering here.

In the usual way, starting from \eqref{Eq:4:1}, we can define $\gm_p(E,Q)$, first for an open set $E\subset Q$, and then for an arbitrary $E\subset Q$. Moreover, it is not hard to check that if $Q_1\subset Q_2$ are open parabolic cylinders and
\[
E\subset \bar E\subset Q_1\subset Q_2,
\]
then
\[
\gm_p(E,Q_1)\ge\gm_p(E,Q_2).
\]
A first important result concerns sets of zero parabolic capacity. {Even though we do not need it in the following, nevertheless, we think it proper to state it here.}
\begin{proposition}
Let $E\subset \bar E\subset Q$. If $\gm_p(E,\rr^{N+1})=0$, then $\gm_p(E,Q)=0$.
\end{proposition}
In the following it will be important to compare the parabolic capacity  we have just defined with the well-known notion of elliptic $p$-capacity. In this respect, as above, let $\Omega\subset\rn$ be an open set, and consider $D\subset\bar D\subset\Omega$. By ${\rm cap}_p(D,\Om)$ we denote the (elliptic) $p$-capacity of $D$ with respect to $\Om$. For its precise definition, and for more details about ${\rm cap}_p(D,\Om)$, we refer to \cite{frehse}. Let $Q=\Omega\times(t_1,t_2)$, and for any set $E\subset\rr^{N+1}$ define $E_\tau=E\cap\{t=\tau\}$. Then, we have the following result.
\begin{proposition}\label{Prop:4:2}
Let $K\subset Q$ be compact. Then,
\begin{equation}\label{eq:4:2}
\gm_p(K,Q)=\int_{t_1}^{t_2}{\rm cap}_p(K_\tau,\Omega)\,d\tau.
\end{equation}
\end{proposition}
\section{An Auxiliary Lemma}
\begin{lemma}\label{Lm:5:1}
Let $Q$, $u_k$, $\mu$, $v$ as in \S~\ref{S:A:1}, consider $\pto\in\Sigma$ with $s<t_o<t$, let
\[
\theta=c\left[\dashint_{K_{2\rho}(x_o)}v(x,t_o)\,dx\right]^{2-p},
\]
where $c$ is the same constant as in \eqref{Eq:3:7:2}, and assume that 
\begin{equation}\label{Eq:5:0}
s\le t_o-\theta\rho^p<t_o+\theta\rho^p\le t. 
\end{equation}
Then, if we let 
\begin{align*}
&\dl(\rho)\df=\frac{{\rm cap}_p(K_{\rho}(x_o)\backslash E,K_{\frac32\rho}(x_o))}{{\rm cap}_p(K_{\rho}(x_o),K_{\frac32\rho}(x_o))},
\end{align*}
there exist constants $\gm_1,\,\gm_2>1$ that depend only on the data $\datap$, such that
\begin{equation}\label{Eq:low-bd1}
\mu\,[\dl(\rho)]^{\frac1{p-1}}\le\gm_1 \dashint_{K_{2\rho}}v(x,t_o)\,dx,
\end{equation}
and
\begin{equation}\label{Eq:low-bd2}
\mu\,[\dl(\rho)]^{\frac1{p-1}}\le\gm_2 \inf_{K_{2\rho}(x_o)} v(\cdot,\tau)\quad\text{ for all }\ \tau\in[t_o+\frac34\theta\rho^p,t_o+\theta\rho^p].
\end{equation}
\end{lemma}
\vskip.2truecm
\begin{remark}
{\normalfont Condition \eqref{Eq:5:0} can always be satisfied, provided $\rho$ is small enough.}
\end{remark}
\begin{proof}
Without loss of generality, we may assume $\pto=(0,0)$. Consider a cut-off function
\[
\zeta(x,t)=\z_1(x)\z_2(t),
\]
with
\begin{equation*}
\z_1(x)=\left\{
\begin{aligned}
&1\quad x\in K_\rho,\\
&0\quad x\in\rn\backslash K_{\frac32\rho},
\end{aligned}
\right.\qquad|D\z_1|\le\frac2\rho,
\end{equation*}
\begin{equation*}
\z_2(t)=\left\{
\begin{aligned}
&1\quad -\frac34\theta\rho^p<t<-\frac14\theta\rho^p,\\
&0\quad t\ge0,\ \ t\le-\theta\rho^p,
\end{aligned}
\right.\qquad|\partial_t\z_2|\le\frac4{\theta\rho^p},
\end{equation*}
and let
\[
Q_1=K_\rho\times(-\frac34\theta\rho^p,-\frac14\theta\rho^p],\qquad
Q_2=K_{\frac32\rho}\times(-\theta\rho^p,0].
\]
If we take $\vp=u_k\z^p$ as test function in the weak formulation of \eqref{Eq:1:5}--\eqref{Eq:1:6},
since $u_k$ is a sub-solution, modulus a Steklov average, we obtain
\begin{align*}
&\iint_{Q_2}\z^p u_k\partial_t u_k\,dxdt+\iint_{Q_2}\bl{A}(x,t,u,Du_k)\cdot D(\z^p u_k)\,dxdt\le0,\\
&\iint_{Q_2}\z^p u_k\partial_t u_k\,dxdt+\iint_{Q_2}\z^p\bl{A}(x,t,u,Du_k)\cdot Du_k\,dxdt\\
&+p\iint_{Q_2}\z^{p-1} u_k \bl{A}(x,t,u,Du_k)\cdot D\z\,dxdt\le0,\\
&\iint_{Q_2}\z^p u_k\partial_t u_k\,dxdt+C_o\iint_{Q_2}\z^p|Du_k|^p\,dxdt\\
&\le pC_1\iint_{Q_2}\z^{p-1} u_k |Du_k|^{p-1} |D\z|\,dxdt.
\end{align*}
If we take into account that $v=\mu-u_k$, (i.e. $u_k=\mu-v$), the previous inequality can be rewritten as
\begin{align*}
&\iint_{Q_2}\z^p (\mu-v)\partial_t (\mu-v)\,dxdt+C_o\iint_{Q_2}\z^p|Dv|^p\,dxdt\\
&\le pC_1\iint_{Q_2}\z^{p-1} u_k |Dv|^{p-1} |D\z|\,dxdt,
\end{align*}
which yields
\begin{align*}
&\iint_{Q_2}\z^p (\mu-v)\partial_t (\mu-v)\,dxdt+C_o\iint_{Q_2}|D(\z v)|^p\,dxdt\\
&\le pC_1\iint_{Q_2}\z^{p-1} u_k |Dv|^{p-1} |D\z|\,dxdt+C_o\iint_{Q_2}v^p|D\z|^p\,dxdt,
\end{align*}
and also
\begin{align*}
&-\iint_{Q_2}\z^p (\mu-v)\partial_t v\,dxdt+C_o\iint_{Q_2}|D(\z v)|^p\,dxdt\\
&\le pC_1\iint_{Q_2}\z^{p-1} u_k |Dv|^{p-1} |D\z|\,dxdt+C_o\iint_{Q_2}v^p|D\z|^p\,dxdt,
\end{align*}
which we rewrite as
\begin{align*}
&\iint_{Q_2}\z^p v\partial_t v\,dxdt+C_o\iint_{Q_2}|D(\z v)|^p\,dxdt\\
&\le pC_1\iint_{Q_2}\z^{p-1} u_k |Dv|^{p-1} |D\z|\,dxdt+C_o\iint_{Q_2}v^p|D\z|^p\,dxdt\\
&+\mu\iint_{Q_2}\z^p \partial_t v\,dxdt.
\end{align*}
Therefore,
\begin{align*}
&\frac12\iint_{Q_2}\z^p \partial_t v^2\,dxdt+C_o\iint_{Q_2}|D(\z v)|^p\,dxdt\\
&\le pC_1\iint_{Q_2}\z^{p-1} u_k |Dv|^{p-1} |D\z|\,dxdt+C_o\iint_{Q_2}v^p|D\z|^p\,dxdt\\
&+{\mu\int_{K_{\frac32\rho}}\z^p v\Big|_{-\theta\rho^p}^0\,dx}+p\mu\iint_{Q_2}\z^{p-1} |\partial_t\z| v\,dxdt,
\end{align*}
where the third term on the right-hand side can be discarded, since it vanishes. Moreover,
\begin{align*}
&{\frac12\int_{K_{\frac32\rho}}\z^p v^2\Big|_{-\theta\rho^p}^0\,dx}+C_o\iint_{Q_2}|D(\z v)|^p\,dxdt\\
&\le pC_1\iint_{Q_2}\z^{p-1} u_k |Dv|^{p-1} |D\z|\,dxdt+C_o\iint_{Q_2}v^p|D\z|^p\,dxdt\\
&+p\mu\iint_{Q_2}\z^{p-1} |\partial_t\z| v\,dxdt+\frac p2\iint_{Q_2}\z^{p-1} v^2|\partial_t \z|\,dxdt,
\end{align*}
where the first term on the left-hand vanishes as well. Hence, writing $v^2=v(\mu-u_k)$, and discarding the resulting negative term,
\begin{align*}
&C_o\iint_{Q_2}|D(\z v)|^p\,dxdt\le pC_1\iint_{Q_2}\z^{p-1} u_k |Dv|^{p-1} |D\z|\,dxdt\\
&+C_o\iint_{Q_2}v^p|D\z|^p\,dxdt+\frac32p\mu\iint_{Q_2}\z^{p-1} |\partial_t\z| v\,dxdt,
\end{align*}
and also
\begin{align*}
&C_o\iint_{Q_2}|D(\z v)|^p\,dxdt\le pC_1\iint_{Q_2}\z^{p-1} u_k |Dv|^{p-1} |D\z|\,dxdt\\
&+\frac{2^p C_o}{\rho^p}\mu\iint_{Q_2}v^{p-1}\,dxdt+6p\mu\left[\sup_{-\theta\rho^p<t<0}\int_{K_{\frac32\rho}}v(x,t)\,dx\right].
\end{align*}
Let us concentrate on the second term on the right-hand side. By the H\"older inequality we have
\begin{align*}
&\frac{2^p C_o}{\rho^p}\mu\iint_{Q_2}v^{p-1}\,dxdt\le\frac{2^p C_o}{\rho^p}\mu\left[\sup_{-\theta\rho^p<t<0}\int_{K_{\frac32\rho}}v^{p-1}(x,t)\,dx\right]\theta\rho^p\\
&\le\frac{2^p C_o}{\rho^p}\mu\left[\sup_{-\theta\rho^p<t<0}\int_{K_{\frac32\rho}}v(x,t)\,dx\right]^{p-1}|K_{\frac32\rho}|^{2-p}\theta\rho^p\\
&\le\gm\theta\mu\rho^N\left[\sup_{-\theta\rho^p<t<0}\dashint_{K_{\frac32\rho}}v(x,t)\,dx\right]^{p-1}.\end{align*}
Taking into account the expression for $\theta$ and Remark~\ref{Rmk:2:1}, we obtain
\begin{align*}
&C_o\iint_{Q_2}|D(\z v)|^p\,dxdt\le pC_1\iint_{Q_2}\z^{p-1} u_k |Dv|^{p-1} |D\z|\,dxdt\\
&+\gm\theta\mu\rho^N\left[\dashint_{K_{2\rho}}v(x,0)\,dx\right]^{p-1}.
\end{align*}
Let us finally consider the first term on the right-hand side. By Lemma~\ref{Lm:A:1:2} we have
\begin{align*}
&pC_1\iint_{Q_2}\z^{p-1} u_k |Dv|^{p-1} |D\z|\,dxdt\le\frac{2pC_1}\rho\mu\iint_{Q_2}|Dv|^{p-1}\,dxdt\\
&\le\frac{\gm}\rho\mu\rho\left[\sup_{-\theta\rho^p<t<0}\int_{K_{\frac54\rho}} v(x,t)dx
+\Big(\frac{\theta\rho^p}{\rho^\lm}\Big)^{\frac1{2-p}}\right]\\
&\le\gm\theta\mu\rho^N\left[\dashint_{K_{2\rho}}v(x,0)\,dx\right]^{p-1}.
\end{align*}
Therefore, we conclude
\begin{equation}\label{Eq:5:1}
\iint_{Q_2}|D(\z v)|^p\,dxdt\le\gm\theta\mu\rho^N\left[\dashint_{K_{2\rho}}v(x,0)\,dx\right]^{p-1}.
\end{equation}
Let us consider the function $\dsty w=\frac{\z v}\mu$: it equals $1$ on $Q_1\backslash E_T$ and vanishes
on the topological boundary of $Q_2$. Therefore, by \eqref{Eq:4:1} and Proposition~\ref{Prop:4:2}
\begin{align*}
\iint_{Q_2}|D(\z v)|^p\,dxdt&\ge\mu^p\gm_p(Q_1\backslash E_T,Q_2)\\
&=\mu^p\int_{-\theta\rho^p}^0{\rm cap}_p(K_{\rho}\backslash E,K_{\frac32\rho})\chi_{(-\frac34\theta\rho^p,-\frac14\theta\rho^p)}(t)dt\\
&=\frac12\mu^p\theta\rho^p{\rm cap}_p(K_{\rho}\backslash E,K_{\frac32\rho}).
\end{align*}
If we substitute back into \eqref{Eq:5:1}
\begin{align*}
\mu^p\theta\rho^p{\rm cap}_p(K_{\rho}\backslash E,K_{\frac32\rho})&\le\gm\theta\mu\rho^N\left[\dashint_{K_{2\rho}}v(x,0)\,dx\right]^{p-1}\\
\mu^{p-1}\frac{{\rm cap}_p(K_{\rho}\backslash E,K_{\frac32\rho})}{{\rm cap}_p(K_{\rho},K_{\frac32\rho})}&\le\gm\left[\dashint_{K_{2\rho}}v(x,0)\,dx\right]^{p-1}.
\end{align*}
\noi If we let
\[
\dl(\rho)\df=\frac{{\rm cap}_p(K_{\rho}\backslash E,K_{\frac32\rho})}{{\rm cap}_p(K_{\rho},K_{\frac32\rho})},
\]
we have
\begin{align*}
\mu[\dl(\rho)]^{\frac1{p-1}}&\le\gm_1\,\dashint_{K_{2\rho}}v(x,0)\,dx,
\end{align*}
and by Proposition~\ref{Prop:3:1} 
\begin{align*}
\mu[\dl(\rho)]^{\frac1{p-1}}&\le\gm_2\inf_{K_{2\rho}\times(\frac34\theta\rho^p,\theta\rho^p]}v.
\end{align*}
\end{proof}
\section{The Proof of Theorem~\ref{Thm:1:1}}\label{S:final}
\subsection{Preliminaries}
Recall the definition of $Q_{R_o}$ and $\om_o$ given in Section~\ref{S:intro}: since $u$ is locally bounded in $E_T$, without loss of generality we may assume that $\om_o\le1$, so that
\[
\tilde Q_1=K_{R_o}(x_o)\times[t_o-2c\,\om_o^{2-p}R_o^p,t_o+c\,\om_o^{2-p}R_o^p]\subset Q_{R_o},
\]
where $c\in(0,1)$ is the same constant as in \eqref{Eq:3:7:2}. It is apparent that 
\[
\om_1=\osc_{\tilde Q_1\cap E_T} u\le\om_o.
\]
Now let
\begin{align*}
g_1^+=\sup_{\tilde Q_1\cap S_T}g,&\qquad g_1^-=\inf_{\tilde Q_1\cap S_T}g,\\
\mu_1^+=\sup_{\tilde Q_1\cap E_T}u,&\qquad\mu_1^-=\inf_{\tilde Q_1\cap E_T}u.
\end{align*}
If the two inequalities
\[
\mu_1^+-\frac{\om_1}4\le g_1^+,\qquad\mu_1^-+\frac{\om_1}4\ge g_1^-
\]
are both true, then subtracting from one another yields
\[
\osc_{\tilde Q_1\cap E_T}u\le 2\osc_{\tilde Q_1\cap S_T}g,
\]
and there is nothing to prove. Therefore, we can assume that 
at least one of the two is violated. 

If the former does not hold, namely
\[
\mu_1^+-\frac{\om_1}4>g_1^+,
\]
then the level
\[
k_1=\mu_1^+-\frac{\om_1}4
\]
is such that Lemma~\ref{Lem:2:0} applies to $u_{k_1}=(u-k_1)_+$,
extended as zero outside $\tilde Q_1\cap E_T$.  Thus
Proposition~\ref{Prop:A:1:1}, Lemma~\ref{Lm:A:1:2}, and
Lemma~\ref{Lm:5:1} can be applied. 



On the other hand, if the latter does not hold, namely 
\[
\mu_1^-+\frac{\om_1}4< g_1^-,
\]
then analogously the level
\[
h_1=\mu_1^-+\frac{\om_1}4
\]
is such that Lemma~\ref{Lem:2:0} applies to $u_{h_1}=(h_1-u)_+$
extended as zero outside $\tilde Q_1\cap E_T$.  As before,
Proposition~\ref{Prop:A:1:1}, Lemma~\ref{Lm:A:1:2}, and
Lemma~\ref{Lm:5:1} can be applied.


\subsection{The First Step}
Consider either $u_{k_1}$ or $u_{h_1}$, and let
\begin{align}\label{Eq:1step}
&\mu_1=\sup_{\tilde Q_1}u_{k_1}\quad\text{ or }\quad\mu_1=\sup_{\tilde Q_1}u_{h_1},\\
&Q_1=K_{2r}(x_o)\times[t_o-c\mu_1^{2-p}2(2r)^p,t_o+c\mu_1^{2-p}(2r)^p],
\end{align}
where $r$ is such that $2r=R_o$.
\subsection{The Induction}
We now proceed by induction: we suppose that we have solved the problem 
up to step $j$ included, namely that we have built 
\begin{equation}\label{Eq:ind}
\begin{aligned}
&r_i=\frac{R_o}{2^{i}},\quad i=1,2,\dots,j,\\
&\tilde Q_i=K_{2r_i}(x_o)\times[t_o-c\mu_{i-1}^{2-p}2(2r_i)^p,t_o+c\mu_{i-1}^{2-p}(2r_i)^p],\quad i=2,\dots,j,\\
&\mu_i^+=\sup_{\tilde Q_i\cap E_T} u,\quad \mu_i^-=\inf_{\tilde Q_i\cap E_T} u,\quad i=2,\dots,j,\\
&\om_i\ge\osc_{\tilde Q_i\cap E_T} u,\quad i=2,\dots,j,\\
&k_i=k_i(\mu^+_i,\om_i),\quad h_i=h_i(\mu^-_i,\om_i),\quad i=2,\dots,j,\\
&\mu_i=\sup_{\tilde Q_i} u_{k_i}\quad\text{ or }\quad \mu_i=\sup_{\tilde Q_i} u_{h_i},\quad \mu_i\le\mu_{i-1},\quad i=2,\dots,j,\\
&Q_i=K_{2r_i}(x_o)\times[t_o-c\mu_{i}^{2-p}2(2r_i)^p,t_o+c\mu_{i}^{2-p}(2r_i)^p],\quad i=2,\dots,j.
\end{aligned}
\end{equation}
If we focus on the last step, by construction, and by the induction assumption, we have
\begin{align*}
&\tilde Q_j=K_{2r_j}(x_o)\times[t_o-c\mu_{j-1}^{2-p}2(2r_j)^p,t_o+c\mu_{j-1}^{2-p}(2r_j)^p],\\
&\mu_j^+=\sup_{\tilde Q_j\cap E_T}u,\quad \mu_j^-=\inf_{\tilde Q_j\cap E_T}u.
\end{align*}
Moreover, at this stage, let
\[
\om_j=\osc_{\tilde Q_j\cap E_T}u.
\]
If both inequalities
\[
\mu^+_j-\frac{\om_j}{4}\le\sup_{\tilde{Q}_j\cap S_T}g\quad\text{and}\quad\mu^-_j+\frac{\om_j}{4}\ge\inf_{\tilde{Q}_j\cap S_T}g
\]
hold true, then subtracting from one another yields
\[
\osc_{\tilde Q_j\cap E_T}u=\om_j\le2\osc_{\tilde{Q}_j\cap S_T}g,
\]
and there is nothing to prove. Therefore, we can assume that at least one of the two is violated. Suppose that 
\[
\mu^+_j-\frac{\om_j}{4}>\sup_{\tilde{Q}_j\cap S_T}g.
\]
Then the level $\dsty k_j=\mu^+_j-\frac{\om_j}{4}$ is such that
Lemma~\ref{Lem:2:0} applies to $u_{k_j}=(u-k_j)_+$, extended as zero
outside $\tilde Q_j\cap E_T$.  Therefore Proposition~\ref{Prop:A:1:1},
Lemmas~\ref{Lm:A:1:2} and \ref{Lm:5:1} can be used.  Finally, let
\begin{equation*}
\begin{aligned}
&\mu_j=\sup_{\tilde Q_j} u_{k_j}\le\mu_{j-1},\\
&Q_j=K_{2r_j}(x_o)\times[t_o-c\mu_{j}^{2-p}2(2r_j)^p,t_o+c\mu_{j}^{2-p}(2r_j)^p],
\end{aligned}
\end{equation*}
where $\mu_j\le\mu_{j-1}$ by the induction hypothesis.

Now, we will show how $\dsty\mu_{j+1}\df=\sup_{\tilde Q_{j+1}} u_k$, satisfies
\[
\mu_{j+1}\le\mu_j,
\]
and we will also give a quantitative upper bound on $\mu_{j+1}$ in terms of $\mu_j$.
For simplicity, in the sequel, unless otherwise stated, we drop the suffix $j$, and simply write $r$, $k$, $\mu$, 
$\tilde Q$, and $Q$.

By construction, $Q\subseteq \tilde Q$. Let $v:Q\to\rr_+$ be defined by $\dsty v=\mu-u_k$. It is apparent that
\[
\sup_Q v\le\mu,\qquad\inf_Q v\ge0.
\]
Since $2^{p+1}>3$, we have
\[
t_o-c\,\mu^{2-p}3r^p>t_o-c\,\mu^{2-p}2(2r)^p.
\]
Moreover, for any $\dsty t\ge t_o-c\,\mu^{2-p}3r^p$, it is not difficult to verify that 
\begin{equation}\label{Eq:6:3}
t-c\left[\dashint_{K_{2r}(x_o)}v(x,t)\,dx\right]^{2-p}r^p\ge t_o-c\,\mu^{2-p}2(2r)^p.
\end{equation}
Consider the closed and bounded interval
\[
I\df=[t_o-c\,\mu^{2-p}3r^p,t_o+c\,\mu^{2-p}r^p]\subset[t_o-c\,\mu^{2-p}2(2r)^p,t_o+c\,\mu^{2-p}(2r)^p],
\]
and the function $f:I\to\rr$ defined by
\[
f(t)=t+c\left[\dashint_{K_{2r}(x_o)}v(x,t)\,dx\right]^{2-p}r^p.
\]
It is straightforward to check that $f\in C^o(I)$. Therefore, $f$ attains all the values of the interval $[\min_I f,\max_I f]$. Notice that
\begin{equation}\label{Eq:6:4}
\begin{aligned}
&\min_I f\le t_o-2c\,\mu^{2-p}r^p,\\
&\max_I f\ge t_o+c\,\mu^{2-p}r^p,\\
&\max_I f\le t_o+c\,\mu^{2-p}2r^p.
\end{aligned}
\end{equation}
Hence $J\df=[t_o-2c\,\mu^{2-p}r^p,t_o+c\,\mu^{2-p}r^p]\subset[\min_I f,\max_I f]$, and we can conclude that $f$ 
attains all values of $J$.

Moreover, by \eqref{Eq:6:3}, we can apply Lemma~\ref{Lm:5:1} and conclude that
\begin{align}
&\mu[\dl(r)]^{\frac1{p-1}}\le\gm_1\dashint_{K_{2r}(x_o)}v(x,\tau)\,dx,\ \ \ \forall\,\tau\in[t_o-c\,\mu^{2-p}3r^p,t_o+c\,\mu^{2-p}r^p],\nonumber\\
&\mu[\dl(r)]^{\frac1{p-1}}\le\gm_2\inf_{K_{2r}(x_o)} v(x,\tau),\ \ \ \forall\,\tau\in[t_o-c\,\mu^{2-p}2r^p,t_o+c\,\mu^{2-p}r^p].\label{Eq:6:6}
\end{align}
If we revert to using the suffix $j$, and we recall that by \eqref{Eq:ind} $r_{j+1}=\frac{r_j}2$, then \eqref{Eq:6:6} yields
\begin{equation}\label{Eq:6:1}
\mu_j[\dl(r_j)]^{\frac1{p-1}}\le\gm_2[\mu_j-\sup_{\tilde Q_{j+1}} u_{k_j}] 
\end{equation}
where
\[
\tilde Q_{j+1}=K_{2r_{j+1}}(x_o)\times[t_o-c\,\mu_j^{2-p}2(2r_{j+1})^p,t_o+c\,\mu_j^{2-p}(2r_{j+1})^p].
\]
Since
\[
\mu_{j+1}=\sup_{\tilde Q_{j+1}} u_k,
\]
we conclude that
\[
\mu_{j+1}\le\left[1-\frac1{\gm_2}\left[\dl(r_{j})\right]^{\frac1{p-1}}\right]\mu_{j},
\]
exactly as we claimed above.

On the other hand, if
\[
\mu^-_j+\frac{\om_j}{4}<\inf_{\tilde{Q}_j\cap S_T}g,
\]
then the level $\dsty h_j=\mu^-_j+\frac{\om_j}{4}$ is such that
Lemma~\ref{Lem:2:0} applies to $u_{h_j}=(h_j-u)_+$ extended as zero
outside $\tilde Q_j\cap E_T$. Therefore Proposition~\ref{Prop:A:1:1},
Lemmas~\ref{Lm:A:1:2} and \ref{Lm:5:1} can be used. Finally, let
\begin{equation*}
\begin{aligned}
&\mu_j=\sup_{\tilde Q_j} u_{h_j}\le\mu_{j-1},\\
&Q_j=K_{2r_j}(x_o)\times[t_o-c\mu_{j}^{2-p}2(2r_j)^p,t_o+c\mu_{j}^{2-p}(2r_j)^p],
\end{aligned}
\end{equation*}
where again $\mu_j\le\mu_{j-1}$ by the induction hypothesis. Working as before, we conclude that 
\begin{equation}\label{Eq:6:1:bis}
\mu_j[\dl(r_j)]^{\frac1{p-1}}\le\gm_2[\mu_j-\sup_{\tilde Q_{j+1}} u_{h_j}],
\end{equation}
which yields 
\[
\mu_{j+1}\le\left[1-\frac1{\gm_2}\left[\dl(r_{j})\right]^{\frac1{p-1}}\right]\mu_{j}.
\]
Notice that, if
\[
\mu^+_j-\frac{\om_j}{4}>\sup_{\tilde{Q}_j\cap S_T}g,
\]
then by construction 
$$\mu_{j}=\sup_{\tilde{Q}_{j}} (u-{k_{j}})_+=\frac{\om_{j}}{4},$$
whereas if
\[
\mu^-_j+\frac{\om_j}{4}<\inf_{\tilde{Q}_j\cap S_T}g,
\]
then by construction 
$$\mu_{j}=\sup_{\tilde{Q}_{j}} (h_j-u)_+=\frac{\om_{j}}{4}.$$
Therefore, as a matter of fact, the nested cylinders we defined are 
\begin{equation*}
\begin{aligned}
&\tilde{Q}_j=K_{2r_j}(x_o)\times[t_o-\frac c4\om_{j-1}^{2-p}2(2r_j)^p, t_o+\frac c4\om_{j-1}^{2-p}(2r_j)^p],\\
&\om_{j}\le\om_{j-1},\\
&Q_j=K_{2r_j}(x_o)\times[t_o-\frac c4\om_{j}^{2-p}2(2r_j)^p, t_o+\frac c4\om_{j}^{2-p}(2r_j)^p],
\end{aligned}
\end{equation*}
and the induction process shows that
\[
\om_{j+1}=\osc_{\tilde{Q}_{j+1}\cap E_T} u\le\om_j.
\]
The quantitative upper bound given in \eqref{Eq:6:1} is actually an upper bound on $\om_{j+1}$, and reads
\begin{equation*}
\begin{aligned}
(\mu^{+}_j-k_j)[\dl(r_j)]^{\frac1{p-1}}&\le\gm_2[\mu^{+}_j-k_j-(\mu^{+}_{j+1}-k_j)]\\
&\le\gm_2(\mu^{+}_j-\mu^{+}_{j+1})\\
&\le\gm_2(\om_j-\osc_{\tilde{Q}_{j+1}\cap E_T}u);
\end{aligned}
\end{equation*}
an analogous result holds for \eqref{Eq:6:1:bis}. Hence, we get 
\begin{equation*}
\osc_{\tilde{Q}_{j+1}\cap E_T}u\le\left[1-\frac1{4\gm_2}\left[\dl(r_{j})\right]^{\frac1{p-1}}\right]\om_{j}.
\end{equation*}
\subsection{The Proof Concluded} 
If we set 
\[
A(r_j)=\frac1{4\gm_2}\left[\dl(r_{j})\right]^{\frac1{p-1}},
\]
and redefine
\begin{equation}\label{iteration}
\om_{j+1}=\max\left\{\left[1-A(r_j)\right]\om_{j},\,2\osc_{\tilde{Q}_j\cap S_T}g\right\},
\end{equation}
we can conclude that
\[
\osc_{\tilde{Q}_{j+1}\cap E_T}u\le\om_{j+1},
\]
and an iteration of the above inequality yields that for a positive integer $m$ 
\begin{align*}
\om_m&\le\om_o\prod_{j=0}^{m-1}[1-A(r_{j})]+2\osc_{\tilde{Q}_o\cap S_T}g,\\
&\le\om_o\prod_{j=0}^{m-1}\exp[-A(r_{j})]+2\osc_{\tilde{Q}_o\cap S_T}g\\
&=\om_o\exp[-\sum_{j=0}^{m-1}A(r_{j})]+2\osc_{\tilde{Q}_o\cap S_T}g.
\end{align*}
For any $y>0$, if $s\in(y,2y)$, then it is easy to check that $\dsty A(s)\le 2^{\frac{N-p}{p-1}} A(2y)$, and therefore,
\begin{align*}
\int_{y}^{2y}A(s)\frac{ds}s&\le 2^{\frac{N-p}{p-1}} A(2y)\\
\int_{2^{-m}r}^{r}A(s)\frac{ds}s&=\sum_{j=0}^{m-1}\int_{2^{-(j+1)}r}^{2^{-j}r}A(s)\frac{ds}s
\le2^{\frac{N-p}{p-1}}\sum_{j=0}^{m-1}A(2^{-j}r)\\
\om_m &\le\om_o\exp\left(-\frac1{2^{\frac{N-p}{p-1}}}\int_{2^{-m}r}^{r}A(s)\frac{ds}s\right)
+2\osc_{\tilde{Q}_o\cap S_T}g.
\end{align*}
Now fix $\rho\in(0,2r_1)$: since the sequence $\{\frac c4\om^{2-p}_{j-1} (2r_j)^p\}$ decreases to $0$ and partitions
the interval $(0,\frac c4\om^{2-p}_o (2r_1)^p)$, there must exist a positive integer $l$ such that
\[
\frac c4\om^{2-p}_{l} (2r_{l+1})^p<\frac c4\om_o^{2-p}\rho^p\le \frac c4\om^{2-p}_{l-1} (2r_l)^p.
\]
Note that from \eqref{iteration} we have
\[
\om_{l}\ge \lm^l\om_o\quad\text{with }\lm=1-\frac1{4\gm_2};
\]
then it is straightforward to verify
\[
Q_{\rho}(\om_o)=K_\rho(x_o)\times[-\frac c4\om_o^{2-p}2\rho^p, \frac c4\om_o^{2-p}\rho^p]\subset \tilde{Q}_l
\]
and
\[
l\ge\frac{p\ln\big(\frac{\rho}r\big)}{(2-p)\ln\lm-p\ln 2}.
\]
Thus
\[
\osc_{Q_{\rho}(\om_o)\cap E_T}u\le\om_o\exp\left\{
-\frac{1}{\gm}\int_{r(\frac{\rho}r)^{\al}}^{r}A(s)\frac{ds}s\right\}+2\osc_{\tilde{Q}_o\cap S_T}g,
\]
where
\[
\al=\frac{p}{(p-2)\frac{\ln\lm}{\ln2}+p},\qquad\gm=2^{\frac{N-p}{p-1}}.
\]
Define the function
\[
(0,1)\ni\rho\to\bar{\om}(\rho)=\exp\left\{-\frac1{\gm}\int_{\rho}^1 A(s)\frac{ds}{s}\right\}.
\]
In the above oscillation estimate choose
\[
r=[\bar{\om}(\rho^\al)]^{\nu}\quad\text{for an arbitrary }0<\nu<1;
\]
we obtain 
\[
\osc_{Q_{\rho}(\om_o)\cap E_T}u\le\om_o\exp\left\{
-\frac{1}{\gm}\int_{\rho^{\al}[\bar{\om}(\rho^\al)]^{\nu(1-\al)}}^{[\bar{\om}(\rho^\al)]^{\nu}}A(s)\frac{ds}s\right\}
+2\osc_{\tilde{Q}_o(\rho)\cap S_T}g.
\]
where 
\begin{equation}\label{large:cyl}
\tilde{Q}_o(\rho)\df=
\left\{
\begin{aligned}
&K_{2r}(x_o)\times[t_o-c\om_o^{2-p}2(2r)^p, t_o+c\om_o^{2-p}(2r)^p]\\ 
&\text{with}\ r=[\bar{\om}(\rho^\al)]^{\nu}.
\end{aligned}
\right.
\end{equation}
Let us now focus on the first term of the right-hand side.
One estimates
\begin{align*}
\int_{\rho^{\al}[\bar{\om}(\rho^\al)]^{\nu(1-\al)}}^{[\bar{\om}(\rho^\al)]^{\nu}}A(s)\frac{ds}s
&=\int_{\rho^{\al}[\bar{\om}(\rho^\al)]^{\nu(1-\al)}}^1A(s)\frac{ds}s
-\int^1_{[\bar{\om}(\rho^\al)]^{\nu}}A(s)\frac{ds}s\\
&\ge -\ln[\bar{\om}(\rho^\al)]+\ln[\bar{\om}(\rho^\al)]^{\nu}\\
&\ge\ln[\bar{\om}(\rho^\al)]^{\nu-1}.
\end{align*}
Therefore, we obtain that 
\begin{align*}
\exp\left\{
-\frac{1}{\gm}\int_{\rho^{\al}[\bar{\om}(\rho^\al)]^{\nu(1-\al)}}^{[\bar{\om}(\rho^\al)]^{\nu}}A(s)\frac{ds}s\right\}
&\le\exp\left\{-\frac1{\gm}\ln[\bar{\om}(\rho^\al)]^{\nu-1}\right\}\\
&=[\bar{\om}(\rho^\al)]^{\frac{1-\nu}{\gm}}.
\end{align*}
Consequently, we conclude that
\[
\osc_{Q_{\rho}(\om_o)\cap E_T}u\le
\om_o\exp\left\{-\frac{1-\nu}{\gm^2}\int_{\rho^\al}^1 A(s)\frac{ds}{s}\right\}
+2\osc_{\tilde{Q}_o(\rho)\cap S_T}g.
\]

\bye